\documentclass[a4paper, 11pt]{amsart}
\usepackage[letterpaper, margin=.99in]{geometry}
\usepackage{latexsym}
\usepackage{amssymb}
\usepackage{amsmath}
\usepackage{amsfonts}
\usepackage{amsthm}
\usepackage{color}
\usepackage{mathtools}
\usepackage{relsize}

            \DeclareFontFamily{U}{wncy}{}
            \DeclareFontShape{U}{wncy}{m}{n}{%
               <5>wncyr5%
               <6>wncyr6%
               <7>wncyr7%
               <8>wncyr8%
               <9>wncyr9%
               <10>wncyr10%
               <11>wncyr10%
               <12>wncyr6%
               <14>wncyr7%
               <17>wncyr8%
               <20>wncyr10%
               <25>wncyr10}{}

\newtheorem{thm}{Theorem}[section]

\newtheorem{lem}[thm]{Lemma}

\theoremstyle{definition}

\newtheorem*{remarks}{Remark}
\newtheorem{example}{Example}

\newcommand{\N}{\mathbb N}
\newcommand{\Z}{\mathbb Z}
\newcommand{\Q}{{\mathbb Q}}
\newcommand{\C}{\mathbb C}

\newcommand{\wh}{\widehat}
\def\al{\alpha}

\def\la{\lambda}

\def\ep{\epsilon}
\def\del{\delta}
\def\be{\beta}
\def\ga{\gamma}

\begin{document}

\raggedbottom

\title[]{Revisiting the action of a subgroup of the modular group on imaginary quadratic number fields}

\author{Abdulaziz Deajim}
\address{Department of Mathematics, King Khalid University,
P.O. Box 9004, Abha, Saudi Arabia} \email{deajim@gmail.com}


\begin {abstract}
Consider the modular group $\mbox{PSL}(2,\mathbb{Z})=\langle x, \, y \,|\, x^2=y^3=1\rangle$ generated by the transformations $x: z\mapsto -1/z$ and $y:z\mapsto (z-1)/z$. Let $H$ be the proper subgroup $\langle y,\,v\,|\, y^3=v^3=1\rangle$ of $\mbox{PSL}(2,\mathbb{Z})$, where $v=xyx$. The reference (M. Ashiq and Q. Mushtaq, {\em Actions of a subgroup of the modular group on an imaginary quadratic field}, Quasigropus and Related Systems {\bf 14} (2006), 133--146) proposed results concerning the action of $H$ on the subset $\{\frac{a+\sqrt{-n}}{c}\,|\, a,b=\frac{a^2+n}{c}, c \in \mathbb{Z}, c\neq 0\}$ of the imaginary quadratic number field $\mathbb{Q}(\sqrt{-n})$ for a positive square-free integer $n$. In the current article, the author points out and corrects errors appearing in the aforementioned reference. Most importantly, the corrected estimate for the number of orbits arising from this action is given, where symmetries in the orbits play a crucial role in the proof.
\end {abstract}
\maketitle

\section{{\bf Introduction}}\label{intro}

The study of the action of the modular group $\mbox{PSL}(2,\Z)$ and its subgroups on various commutative algebraic structures has been common since the works of Graham Higman and his collaborators in the 1980s (\cite{HM}, \cite{Mush0}, \cite{Mush1}). Moreover, various tracks of applications of the action of the modular group (especially in multimedia security) have been appearing recently, see for instance \cite{A}, \cite{R1}, \cite{R2}, \cite{R3}, and \cite{Sh}. This indeed suggests a promising future for the applications of this seemingly purely mathematical topic.

Stothers in \cite{Sto} showed that the modular group $G=\mbox{PSL}(2, \Z)$ has the finite presentation $\langle x,\,y\,|\, x^2=y^3=1\rangle$, where $x$ and $y$ are, respectively, the linear fractional transformations $z \mapsto -1/z$ and $z \mapsto (z-1)/z$. For a square-free integer $m$, the action of $G$ and its subgroups on certain subsets of the quadratic number field $\Q(\sqrt{m})$ has been extensively studied (see for instance \cite{C etal}, \cite{D}, \cite{DA}, \cite{MR}, \cite{MZ}, \cite{Mush1}, \cite{Mush2}, \cite{R}, and \cite{Y}). In particular, Ashiq and Mushtaq (in \cite{AM1}) studied the proper subgroup $H$ of $G$ generated by the transformations $y$ and $v=xyx: z \mapsto -1/(z+1)$, where they showed that $y^3=v^3=1$ are defining relations for $H$, so $H=\langle y,\,v\,|\,y^3=v^3=1\rangle$.

For a positive square-free integer $n$, consider the following subset of the imaginary quadratic number field $\Q(\sqrt{-n})$:
$$\Q^*(\sqrt{-n}):=\{\frac{a+\sqrt{-n}}{c}\,|\, a,b=\frac{a^2+n}{c}, c \in \Z, c\neq 0\}.$$
Deajim and Aslam in \cite{DA} studied the action of $G$ on $\Q^*(\sqrt{-n})$, in which the main result is counting the orbits resulting from this action. Ashiq and Mushtaq in \cite{AM2} studied the action of the subgroup $H$ on $\Q^*(\sqrt{-n})$, where they suggested a formula to count the orbits resulting from this action. However, the author of this article has found some technical/mathematical mistakes in \cite{AM2} which invited the need to revisit its findings with the aim of correcting these mistakes and, particularly, giving the correct formula for the number of orbits.

In Section 2, the errors of \cite{AM2} are listed along with the reasoning for such a judgment. Section 3 aims primarily at correcting the main result of \cite{AM2}, namely \cite[Theorem 9]{AM2}, by giving in Theorem \ref{orbits} the accurate formula for the number of orbits arising from the action of $H$ on $\Q^*(\sqrt{-n})$. Besides, some lemmas which were used to prove Theorem \ref{orbits} help also in correcting some of the mistakes of \cite{AM2} mentioned in Section 2. Finally, a conclusion containing a summary of the main findings of the article is given in Section 4.

\section{{\bf A summary of errors in \cite{AM2}}}
Throughout this article, $n$ denotes a positive square-free integer.

An element $\al=(a+\sqrt{-n})/c\in \Q^*(\sqrt{-n})$ is said to be totally positive if $ac>0$, totally negative if $ac<0$, and of norm zero if $||\al||:=|a|=0$. For $b=(a^2+n)/c$, it is obvious that $bc>0$ and so $b$ and $c$ are always of the same sign.

For every $\al\in \Q^*(\sqrt{-n})$, denote the $H$-orbit $\{h(\al)\,|\, h\in H\}$ by $\al^H$ and the $G$-orbit $\{g(\al)\,|\, g\in G\}$ by $\al^G$. It is clear that $\al^H\subseteq \al^G$. Denote the set of $H$-orbits (resp. the set of $G$-orbits) in $\Q^*(\sqrt{-n})$ by $\mathcal{O}^H_{-n}$ (resp. $\mathcal{O}^G_{-n}$).

As usual, we denote by $\lfloor \; . \; \rfloor$ the floor function, by $d(k)$ the number of positive divisors of $k$, and by $d_{\leq i}(k)$, for $i,k\in \N$ with $i\leq k$, the number of positive divisors of $k$ which do not exceed $i$. For example, $d(15)=4$ and $d_{\leq 5}(15)=3$. Recall that $d$ is a multiplicative function and, for a prime $p$, $d(p^t)=t+1$ where $t\in \N$.

For $\al=(a+\sqrt{-n})/c\in \Q^*(\sqrt{-n})$, we sometimes use the notation $a_\al$, $b_\al$, and $c_\al$ for $a$, $b$, and $c$ respectively. Throughout the article and without mention, we shall make use of the following table which shows the effect of the action of $t\in \{x, y, y^2, v, v^2\}$ on $\al$ and can be verified through straightforward computations:\\
\begin{center}
\begin{tabular}{c|| c c c c c}
\hline
 $t$ & $a_{t(\al)}$ & \quad & $b_{t(\al)}$ & \quad & $c_{t(\al)}$\\
\hline
$x$ & $-a$ & $\quad$ & $c$ & $\quad$ & $b$\\
$y$ & $-a+b$ & $\quad$ & $-2a+b+c$ & $\quad$ & $b$\\
$y^2$ & $-a+c$ & $\quad$ & $c$ & $\quad$ & $-2a+b+c$\\
$v$ & $-a-c$ & $\quad$ & $c$ & $\quad$ & $2a+b+c$\\
$v^2$ & $-a-b$ & $\quad$ & $2a+b+c$ & $\quad$ & $b$\\
\hline
\end{tabular}
\end{center}
\begin{center} \tablename{$\;$1}: The action of $x,y,y^2,v,$ and $v^2$ \end{center}

\vspace{.5cm}
We summarize below the main errors in \cite{AM2} with some comments on them. \\

\begin{itemize}
\item[{\bf Claim 1}] \cite[line 1 of the proof of Theorem 6]{AM2}:\begin{quote} {\it Let $\al$ be a totally positive imaginary quadratic number. Then, by Theorem 3 (i), $y(\al)$ or $y^2(\al)$ is a totally negative imaginary quadratic number.\\} \end{quote}

Claim 1 is not necessarily true. We show in Lemma \ref{+,-,0} that if $\al$ is totally positive, then one of $y(\al)$ and $y^2(\al)$ is totally positive and the other is either totally positive, totally negative, or of norm zero. For instance, consider the totally positive element $\al=(2+\sqrt{-5})/3$. Then, both $y(\al)=(1+\sqrt{-5})/3$ and $y^2(\al)=(1+\sqrt{-5})/2$ are totally positive as well. Another example is given by $\be=(1+\sqrt{-1})/2$, then $y(\be)=\sqrt{-1}$ is of norm zero and $y^2(\be)=1+\sqrt{-1}$ is totally positive.
On the other hand, it should also be noted that \cite[Theorem 3 (i)]{AM2} states that $y(\al)$ and $y^2(\al)$ are totally positive when $\al$ is totally negative, which is true, but Claim 1 above strangely misquotes it. Note that Claim 1 being false renders the proof of \cite[Theorem 6]{AM2} false.\\

\item[{\bf Claim 2}] \cite[Theorem 8]{AM2}:\begin{quote} {\it (v) If $\al=(1+\sqrt{-n})/c_1$ ($n\neq 3$), where $1+n=c_1c_2$, $c_1\neq 1$ or $n+1$, then $\al$ is the only element of norm 1 in $\al^H$.\\} \end{quote}

It is strange that their proof of (v) contradicts the statement above, where the authors of \cite{AM2} stated explicitly that $y^2(\al)=(1+\sqrt{-n})/\frac{(n+1)}{2}$ when $c_1=2$. But then it is obvious that $y^2(\al)\in \al^H$ and is of norm 1 with $\al\neq y^2(\al)$! This shows that the statement of (v) is false in general. For instance, consider $\al=(1+\sqrt{-5})/2$. Then $\al$ is of norm 1 and satisfies the assumptions of Theorem 8 (v) mentioned above. However, $y^2(\al)=(1+\sqrt{-5})/3$ is also of norm 1 belonging to $\al^H$ and is different from $\al$.

As for the other parts of \cite[Theorem 8]{AM2}, it is also noted that the arguments in the proof of parts (iii) and (iv) of \cite[Theorem 8]{AM2} are not sufficient although the statements are correct. In fact, parts (iii) and (iv) specify certain cases where an $H$-orbit contains only one element of norm zero. Part of the proof of Theorem \ref{orbits} shows that this is always the case for any $n$ and for any $H$-orbit containing an element of norm zero.\\

\item[{\bf Claim 3}] \cite[Theorem 9]{AM2}: \begin{quote}{\it If $n\neq 3$, then the total number of orbits of $\Q^*(\sqrt{-n})$ under the action of $H$ are: \begin{itemize} \item[(i)] $2[d(n)+2d(n+1)-6]$ if $n$ is odd, and
\item[(ii)] $2[d(n)+2d(n+1)-4]$ if $n$ is even. \\ \end{itemize}}
\end{quote}

This is the main result of \cite{AM2}, which is inaccurate. Its proof relies on \cite[Theorem 8]{AM2}, which has some gaps (see Claim 2). The correct estimate of the number of orbits is given in Theorem \ref{orbits} below. We briefly give here some counter-examples to \cite[Theorem 9]{AM2}. For odd integers, consider $n=1$ and $n=21$ for instance. For $n=1$, the above estimate gives $|\mathcal{O}^H_{-1}|=-2$ which is absurd. Our estimate (see Theorem \ref{orbits}) gives $|\mathcal{O}^H_{-1}|=2$. For $n=21$, the above estimate gives $|\mathcal{O}^H_{-21}|=12$, whereas our estimate gives $|\mathcal{O}^H_{-21}|=16$. In fact, according to Theorem \ref{orbits}, the number of orbits in this case has to be congruent to zero modulo 8. For even integers, consider $n=26$ for instance. The above estimate gives $|\mathcal{O}^H_{-26}|=16$, whereas our estimate gives $|\mathcal{O}^H_{-26}|=24$. The details of the computations of these counter-examples are given just after the proof of Theorem \ref{orbits}.\\

\end{itemize}

\section{{\bf Lemmas and Main Result}}
The main result of the article is to prove Theorem \ref{orbits} which gives the precise estimate for the number of $H$-orbits $|\mathcal{O}^H_{-n}|$, correcting \cite[Theorem 9]{AM2}. Besides serving to prove Theorem \ref{orbits}, some of the lemmas below contribute also in correcting some of the errors of \cite{AM2} mentioned in Section 2.

\begin{lem}\label{+,-,0}
Let $\al=(a+\sqrt{-n})/c\in \Q^*(\sqrt{-n})$.
\begin{enumerate}
\item[(1)] If $\al$ is totally positive, then one of $y(\al)$ and $y^2(\al)$ is totally positive and the other is either totally positive, totally negative, or of norm zero; whereas both $v(\al)$ and $v^2(\al)$ are totally negative.
\item[(2)] If $\al$ is totally negative, then one of $v(\al)$ and $v^2(\al)$ is totally negative and the other is either totally negative, totally positive, or of norm zero; whereas both $y(\al)$ and $y^2(\al)$ are totally positive.
\item[(3)] If $\al$ is of norm zero, then both $y(\al)$ and $y^2(\al)$ are totally positive and both $v(\al)$ and $v^2(\al)$ are totally negative.
\end{enumerate}
\end{lem}

\begin{proof} Once and for all, note here and elsewhere that Table 1 is used without mention.
\begin{enumerate}
\item[(1)] Assume that $a,b,c>0$ (the case $a,b,c<0$ is proved similarly). We show first that at least one of $y(\al)$ and $y^2(\al)$ is totally positive. To the contrary, suppose that neither of them is totally positive. If $y(\al)$ is totally negative, then (as $c_{y(\al)}=b>0$) $a_{y(\al)}=-a+b<0$. As $b_{y^2(\al)}=c>0$ and as $b_{y^2(\al)}$ and $c_{y^2(\al)}$ have the same sign, we have $-a+b-a+c=c_{y^2(\al)}>0$ implying that $a_{y^2(\al)}=-a+c>0$ (as $-a+b<0$). So, $y^2(\al)$ is totally positive, a contradiction. Similarly, if $y^2(\al)$ is totally negative, then $y(\al)$ is totally positive, a contradiction. On the other hand, if $y(\al)$ is of norm zero, then $a_{y(\al)}=-a+b=0$. So, $c_{y^2(\al)}=-a+b-a+c=-a+c=a_{y^2(\al)}$ implying that (as $c_{y^2(\al)}\neq 0$) $a_{y^2(\al)}c_{y^2(\al)}>0$ and thus $y^2(\al)$ is totally positive, a contradiction. Similarly, if $y^2(\al)$ is of norm zero, then $y(\al)$ is totally positive, a contradiction. This argument leads to the assertion that at least one of $y(\al)$ and $y^2(\al)$ is totally positive. Suppose that it is $y(\al)$ which is totally positive (if it were $y^2(\al)$, then the proof would be similar). It then follows that $y^2(\al)$ is either totally positive, totally negative, or of norm zero depending respectively on whether $a_{y^2(\al)}>0$, $a_{y^2(\al)}<0$, or $a_{y^2(\al)}=0$. As for $v(\al)$ and $v^2(\al)$, note that having $a_{v(\al)}=-a-c<0$ and $c_{v(\al)}=2a+b+c>0$ implies that $v(\al)$ is totally negative. Also, having $a_{v^2(\al)}=-a-b<0$ and $c_{v^2(\al)}=b>0$ implies that $v^2(\al)$ is totally negative.
\item[(2)] Similar to the proof of part (1).
\item[(3)] Suppose that $a=0$. Then $a_{y(\al)}c_{y(\al)}=b^2>0$, $a_{y^2(\al)}c_{y^2(\al)}=bc+c^2>0$ ($b$ and $c$ are of the same sign), and so $y(\al)$ and $y^2(\al)$ are totally positive. On the other hand, $a_{v(\al)}c_{v(\al)}=-bc-c^2<0$, $a_{v^2(\al)}c_{v^2(\al)}=-b^2<0$, and so $v(\al)$ and $v^2(\al)$ are totally negative.
\end{enumerate}
\end{proof}

We have the following obvious observations (see \cite[Lemmas 3.2 and 3.3]{DA}).
\begin{lem}\label{x}$($\cite[Lemmas 3.2 and 3.3]{DA}$)$ For $\al=(a+\sqrt{-n})/c\in \Q^*(\sqrt{-n})$, we have
\begin{itemize}
\item[(i)] $\al$ is totally positive if and only if $x(\al)$ is totally negative.
\item[(ii)] $\al$ has norm zero if and only if $x(\al)$ has norm zero.
\end{itemize}
\end{lem}

For $\al \in \Q^*(\sqrt{-n})$, the sets $\stackrel{y}{\wh{\;\al\;}}=\{\al, y(\al), y^2(\al)\}$ and $\stackrel{v}{\wh{\;\al\;}}=\{\al, v(\al), v^2(\al)\}$ are called respectively the {\it $y$-cycle} and the {\it $v$-cycle} of $\al$ (or just {\it cycles} when $\al$, $y$, and $v$ are clear in the context). We call a cycle totally positive (resp. totally negative) if all its elements are totally positive (resp. if all its elements are totally negative).

\begin{lem}\label{cycles}
Under the action of $H$ on $\Q^*(\sqrt{-n})$, the following holds for $\al=(a+\sqrt{-n})/c\in \Q^*(\sqrt{-n})$:
\begin{itemize}
\item[(i)] \cite[Lemma 3.3]{DA} The cycle $\stackrel{y}{\wh{\;\al\;}}$ is totally positive if and only if either $($$a>0$, $a<b$, $a<c$$)$ or $($$a<0$, $a>b$, $a>c$$)$.
\item[(ii)] The cycle $\stackrel{v}{\wh{\;\al\;}}$ is totally negative if and only if either $($$a>0$, $-a>b$, $-a>c$$)$ or $($$a<0$, $-a<b$, $-a<c$$)$.
\item[(iii)] The cycle $\stackrel{y}{\wh{\;\al\;}}$ is totally positive if and only if the cycle $\stackrel{v}{\,\wh{x(\al)\,}}$ is totally negative.
\item[(iv)] If the cycle $\stackrel{y}{\wh{\;\al\;}}$ is totally positive, then $\al^H\neq x(\al)^H$ (i.e. $\al$ and $x(\al)$ belong to distinct $H$-orbits).
\item[(v)] If $n\neq 1$ and $\al$ has norm zero, then $\al^H\neq x(\al)^H$.
\end{itemize}
\end{lem}

\begin{proof}
\begin{itemize}
\item[(i)] This is \cite[Lemma 3.3]{DA}.
\item[(ii)] Suppose that $\stackrel{v}{\wh{\;\al\;}}$ is totally negative. As $\al$ is totally negative, either ($a>0$ and $b,c<0$) or ($a<0$ and $b,c>0$). Assume that $a>0$ and $b,c<0$. Since $b_{v(\al)}=c<0$, we get $c_{v(\al)}<0$. As further $v(\al)$ is totally negative, we must then have $-a-c=a_{v(\al)}>0$ and so $-a>c$. Since also $c_{v^2(\al)}=b<0$ and $v^2(\al)$ is totally negative, $a_{v^2(\al)}=-a-b>0$ and so $-a>b$. A similar argument yields that if $a<0$ and $b,c>0$ then $-a<b$ and $-a<c$.

    Conversely, suppose that $a>0$, $-a>b$, and $-a>c$ (the case when $a<0$, $-a<b$, and $-a<c$ is handled similarly). Since $a>0$ and $-a>c$, $c<0$ and so $\al$ is totally negative. As $a_{v(\al)}=-a-c>0$ and $b_{v(\al)}=c<0$ and $b_{v(\al)}, c_{v(\al)}$ are of the same sign, $v(\al)$ is totally negative too. As $a_{v^2(\al)}=-a-b>0$ and $c_{v^2(\al)}=b<0$, $v^2(\al)$ is totally negative as well. Hence, the cycle $\stackrel{v}{\wh{\;\al\;}}$ is totally negative.
\item[(iii)] Suppose that $\stackrel{y}{\wh{\;\al\;}}$ is totally positive. By (i), if $a>0$, $a<b$, and $a<c$, then it follows from (ii) that $\stackrel{v}{\,\wh{x(\al)\,}}$ is totally negative since $a_{x(\al)}=-a<0$, $-a_{x(\al)}=a<c=b_{x(\al)}$, and $-a_{x(\al)}=a<b=c_{x(\al)}$. The case when $a<0$, $a>b$, and $a>c$ is similar.

    Conversely, suppose that $\stackrel{v}{\,\wh{x(\al)\,}}$ is totally negative. Then by (ii), if $a_{x(\al)}=-a>0$, $-a_{x(\al)}=a>b_{x(\al)}=c$, and $-a_{x(\al)}=a>c_{x(\al)}=b$, then it follows from (i) that $\stackrel{y}{\wh{\;\al\;}}$ is totally positive. The case when $a_{x(\al)}<0$, $-a_{x(\al)}<b_{x(\al)}$, and $-a_{x(\al)}<c_{x(\al)}$ is similar.
\item[(iv)] Suppose that $\stackrel{y}{\wh{\;\al\;}}$ is totally positive. By (iii), $\stackrel{v}{\,\wh{x(\al)\,}}$ is totally negative. To show that $\al^H\neq x(\al)^H$, we need to show that $h(\al)\neq x(\al)$ for any $h\in H$. It can be seen that an arbitrary non-identity element of $H$ takes one and only one of the following forms:
    \begin{align*}
    h_1&=y^{\ep}, \;\mbox{for}\; \ep\in \{1,2\} \\
    h_2&=v^{\del}, \;\mbox{for}\; \del\in \{1,2\}\\
    h_3&=y^{\ep_k}v^{\del_k} \dots y^{\ep_2}v^{\del_2}y^{\ep_1}v^{\del_1}, \;\mbox{for}\; k\geq 1, \ep_i, \del_i\in \{1,2\}, 1\leq i \leq k\\
     h_4&=v^{\del_k}y^{\ep_{k-1}}v^{\del_{k-1}} \dots y^{\ep_2}v^{\del_2}y^{\ep_1}v^{\del_1}, \;\mbox{for}\; k\geq 2, \ep_i, \del_j\in \{1,2\}, 1\leq i \leq k-1, 1\leq j \leq k\\
      h_5&=v^{\del_k}y^{\ep_k} \dots v^{\del_2}y^{\ep_2}v^{\del_1}y^{\ep_1}, \;\mbox{for}\; k\geq 1, \ep_i, \del_i\in \{1,2\}, 1\leq i \leq k\\
      h_6&=y^{\ep_k}v^{\del_{k-1}}y^{\ep_{k-1}} \dots v^{\del_2}y^{\ep_2}v^{\del_1}y^{\ep_1}, \;\mbox{for}\; k\geq 2, \ep_i, \del_j\in \{1,2\}, 1\leq i \leq k, 1\leq j \leq k-1
    \end{align*}
\begin{itemize}
\item[\underline{The $h_1$ case:}] Since there is no $z\in \C$ such that $y^\ep(z)=x(z)$, we get $h_1(\al) = y^\ep(\al)\neq x(\al)$. Alternatively, $y^\ep(\al)\neq x(\al)$ since $y^\ep(\al)$ is totally positive and $x(\al)$ is totally negative (by Lemma \ref{x}).
\item[\underline{The $h_2$ case:}] Since there is no $z\in \C$ such that $v^\del(z)=x(z)$, we get $h_2(\al) = v^\del(\al)\neq x(\al)$. Alternatively, $v^\del (\al)\neq x(\al)$ since if equality holds then $xy^\del x(\al)=x(\al)$ would imply that $x(\al)=y^{-\del}(\al)$, a contradiction (see the $h_1$ case above).
\item[\underline{The $h_3$ case:}] As $\al$ is totally positive, it follows from Lemma \ref{+,-,0} (1) that $v^{\del_1}(\al)$ is totally negative and so, by Lemma \ref{+,-,0} (2), $y^{\ep_1}v^{\del_1}(\al)$ is totally positive. Repeating this argument for $i=1, \dots, k$, we get that $h_3(\al)$ is totally positive. Since $x(\al)$ is totally negative, $h_3(\al)\neq x(\al)$.
\item[\underline{The $h_4$ case:}] If we let $h_4^*=v^{-\del_k}h_4=y^{\ep_{k-1}}v^{\del_{k-1}} \dots y^{\ep_2}v^{\del_2}y^{\ep_1}v^{\del_1}$, then by an argument similar to that in the $h_3$ case, $h_4^*(\al)$ is totally positive. If $h_4(\al)=x(\al)$, then $h_4^*(\al)=v^{-\del_k}x(\al)\in \stackrel{v}{\,\wh{x(\al)\,}}$ and so $h_4^*(\al)$ is totally negative, a contradiction.
\item[\underline{The $h_5$ case:}] If we let $h_5^*=v^{-\del_k}h_5=y^{\ep_k} \dots v^{\del_2}y^{\ep_2}v^{\del_1}y^{\ep_1}$, then as $y^{\ep_1}(\al)$ is totally positive, it follows by an argument similar to that in the $h_3$ case that $h_5^*(\al)$ is totally positive. If $h_5(\al)=x(\al)$, then $h_5^*(\al)=v^{-\del_k}x(\al)$, which is impossible since $h_5^*(\al)$ is totally positive, while $v^{-\del_k}x(\al)$ is totally negative (see the $h_4$ case).
\item[\underline{The $h_6$ case:}] If we let $h_6^*=y^{-\ep_k}h_6=v^{\del_{k-1}}y^{\ep_{k-1}} \dots v^{\del_2}y^{\ep_2}v^{\del_1}y^{\ep_1}$, then as $y^{\ep_1}(\al)$ is totally positive, it follows by an argument similar to that in the $h_3$ case that $h_6^*(\al)$ is totally negative. If $h_6(\al)=x(\al)$, then $h_6^*(\al)=y^{-\ep_k}x(\al)$, which is impossible since $h_6^*(\al)$ is totally negative, while $y^{-\ep_k}x(\al)$ is totally positive by Lemma \ref{+,-,0} (2).
\end{itemize}
Now, by having just shown the impossibility of all cases above, we conclude that there is no $h\in H$ such that $h(\al)=x(\al)$ and therefore $\al$ and $x(\al)$ belong to distinct $H$-orbits as claimed.
\item[(v)] Remark first that we excluded the case $n=1$ since the only elements of $\C$ fixed by $x$ are $\pm \sqrt{-1} \in \Q^*(\sqrt{-1})$ (i.e. $x(\pm \sqrt{-1})=\pm \sqrt{-1}$); see \cite[Lemma 3.4]{DA}. Let $n \neq 1$ and suppose that $\al$ has norm zero. Then, $x(\al)$ has norm zero too. We need to show that $h(\al) \neq x(\al)$ for any $h\in H$. By Lemma \ref{+,-,0} (3), both $y(\al)$ and $y^2(\al)$ are totally positive, and both $v(\al)$ and $v^2(\al)$ are totally negative. Now, for $h_1, \dots, h_6$ as in (iv), and since none of $y(\al), \, y^2(\al),\, v(\al), v^2(\al)$ has norm zero it is easily seen that $h_i(\al)$ is either totally negative or totally positive for every $i=1, \dots, 6$ (see the argument in (iv)). Thus $h_i(\al)\neq x(\al)$ for any $i=1, \dots, 6$. This settles the claim.
\end{itemize}

\end{proof}

\begin{example}
For $n=7$ and $\al=(-1+\sqrt{-7})/(-4)$, it can be checked that $y(\al)=(-1+\sqrt{-7})/(-2)$, $y^2(\al)=(-3+\sqrt{-7})/(-4)$, and so $\al$, $y(\al)$, and $y^2(\al)$ are all totally positive. On the other hand, $x(\al)=(1+\sqrt{-7})/(-2)$, $vx(\al)=(1+\sqrt{-7})/(-4)$, and $v^2x(\al)=(3+\sqrt{-7})/(-4)$ which are all totally negative.
\end{example}

\begin{lem}\label{orbits lemma} $($\cite[Corollary 3.1]{DA}$)$
Under the action of $G$ on $\Q^*(\sqrt{-n})$, we have the following:
\begin{itemize}
\item[(i)] Every $G$-orbit in $\Q^*(\sqrt{-1})$ contains a unique element of norm zero.
\item[(ii)] Every $G$-orbit in $\Q^*(\sqrt{-2})$ contains a unique pair of distinct elements of norm zero.
\item[(iii)] Every $G$-orbit in $\Q^*(\sqrt{-n})$, for $n\geq 3$, contains either a unique pair of distinct elements of norm zero or a unique totally positive cycle, but not both.
\end{itemize}
\end{lem}

The aim of our main result below is to compute the cardinality $|\mathcal{O}^H_{-n}|$ correcting the estimate given in \cite[Theorem 9]{AM2} (see Claim 3 in Section 2). In particular, we utilize the above results to simply show that $$\mbox{$|\mathcal{O}^H_{-1}|=|\mathcal{O}^G_{-1}|$, and $|\mathcal{O}^H_{-n}|=2|\mathcal{O}^G_{-n}|$ for $n\geq 2$,}$$ and then make use of the precise estimate of $|\mathcal{O}^G_{-n}|$ given in \cite[Theorem 2.1]{DA} for all $n\geq 1$.

\vspace{.3cm}
\begin{thm}\label{orbits}
Under the action of $H$, the number of orbits in $\Q^*(\sqrt{-n})$ is

$\qquad\qquad|\mathcal{O}^H_{-n}|
= \begin{cases}
|\mathcal{O}^G_{-1}| & \mbox{, if $n=1$}\\\\
2|\mathcal{O}^G_{-n}| & \mbox{, if $n\geq 2$}\\
\end{cases}$

$$\;\,= \begin{cases}
2 & \mbox{, if $n=1$}\\
4 & \mbox{, if $n=2$}\\
8 & \mbox{, if $n=3$} \\
2d(n)+ \frac{4}{3}\;\sum_{i=1}^{\lfloor (n-1)/2 \rfloor}[ d(i^2+n)-2d_{\leq i} (i^2+n)] & \mbox{, if $n>3$.}
\end{cases}$$
Moreover, $|\mathcal{O}^H_{-n}| \equiv 0\;(\bmod \,8)$ for $n \geq 3$.
\end{thm}

\begin{proof}
We begin with the case $n\geq 2$, where we show that $|\mathcal{O}_{-n}^H|=2|\mathcal{O}_{-n}^G|$. We show this equality by showing that every $G$-orbit splits into two distinct $H$-orbits. Let $\al^G\in \mathcal{O}_{-n}^G$. We know from Lemma \ref{orbits lemma} (ii, iii) that $\al^G$ either contains a unique pair of distinct elements of norm zero or a unique totally positive cycle, but not both. We proceed by dealing with these two possibilities.

Firstly, suppose that $\al^G$ contains a unique pair of distinct elements $\be$ and $x(\be)$ of norm zero. Since $\be\neq x(\be)$ in this case, we get from Lemma \ref{cycles} (v) that it cannot be the case that both elements belong to $\al^H$. We show that $\al^G$ splits into two distinct $H$-orbits ($\al^H$ and $\be^H$) or ($\al^H$ and $x(\be)^H$) by showing that $\al^H$ must contain one and only one of these two norm-zero elements while the other belongs to $\al^G-\al^H$. If $\be\in \al^H$, then (by Lemma \ref{cycles} (v)) $\al^G$ is the disjoint union of $\al^H$ and $x(\be)^H$ and we are done. Similarly, if $x(\be)\in \al^H$, then $\al^G$ is the disjoint union of $\al^H$ and $\be^H$. Now, it remains to show that one of $\be$ or $x(\be)$ must belong to $\al^H$. Assume that $\be\not\in \al^H$; we show that $x(\be)\in \al^H$. Since $\al\in \al^G-\al^H$, there is some $g\in G-H$ such that $\be=g(\al)$. In order to show that $x(\be)\in \al^H$, we argue that $xg\in H$, which then would imply that $x(\be)=xg(\al)\in \al^H$ as desired. Keeping in mind the forms $h_0=1$ and $h_1, \dots, h_6$ that elements of $H$ can take (see the proof of Lemma \ref{cycles} (iv)), we see that $g$, as an element of $G$ but not an element of $H$, can take one of the forms $xh_i$, $h_i x$, $y^\ga x h_i$, or $h_i x y^\ga$ for an appropriate $h_i$, $i=0, \dots, 6$, $\ga\in \{1,2\}$. We do the following argument that exhausts all possible cases.

If $g=xh_i$, then $xg=h_i\in H$ for $i=0, \dots, 6$.

If $g=y^\ga x h_i$, then $xg=xy^\ga x h_i=v^\ga h_i\in H$ for $i=0, \dots, 6$.

If $g=h_0 x=x$, then $xg=1\in H$.

If $g=h_0xy^\ga$, then $xg=y^\ga\in H$.

If $g=h_1 x=y^\ep x$, then $xg=xy^\ep x = v^\ep\in H$.

If $g=h_1 x y^\ga=y^\ep x y^\ga$, then $xg=xy^\ep x y^\ga=v^\ep y^\ga\in H$.

If $g=h_2 x=v^\del x=xy^\del$, then $xg=y^\del\in H$.

If $g=h_2 xy^\ga=v^\del x y^\ga=xy^{\del+\ga}$, then $xg=y^{\del+\ga}\in H$.

If $g=h_3 x$, then $xg= x(y^{\ep_k}v^{\del_k} \dots y^{\ep_1}v^{\del_1}) x=v^{\ep_k} y^{\del_k} \dots v^{\ep_1} y^{\del_1}\in H$.

If $g=h_3 x y^\ga$, then $xg=x(y^{\ep_k}v^{\del_k} \dots y^{\ep_1}v^{\del_1})xy^\ga=v^{\ep_k}y^{\del_k}\dots v^{\ep_1} y^{\del_1+\ga}\in H$.

If $g=h_4 x$, then $xg=x(v^{\del_k}y^{\ep_{k-1}}v^{\del_{k-1}} \dots y^{\ep_1}v^{\del_1})x=y^{\del_k}v^{\ep_{k-1}}y^{\del_{k-1}}\dots v^{\ep_1} y^{\del_1}\in H$.

If $g=h_4 xy^\ga$, then $xg=x(v^{\del_k}y^{\ep_{k-1}}v^{\del_{k-1}} \dots y^{\ep_1}v^{\del_1})xy^\ga= y^{\del_k}v^{\ep_{k-1}} y^{\del_{k-1}} \dots v^{\ep_1} y^{\del_1+\ga}\in H$.

If $g=h_5 x$, then $xg=x(v^{\del_k}y^{\ep_k} \dots v^{\del_1}y^{\ep_1})x=y^{\del_k}v^{\ep_k}\dots y^{\del_1}v^{\ep_1}\in H$.

If $g=h_5xy^\ga$, then $xg=x(v^{\del_k}y^{\ep_k} \dots v^{\del_1}y^{\ep_1})xy^\ga=y^{\del_k} v^{\ep_k} \dots y^{\del_1}v^{\ep_1}y^\ga\in H$.

If $g=h_6 x$, then $xg=x(y^{\ep_k}v^{\del_{k-1}}y^{\ep_{k-1}} \dots v^{\del_1}y^{\ep_1})x=v^{\ep_k}y^{\del_{k-1}} v^{\ep_{k-1}}\dots y^{\del_1}v^{\ep_1}\in H$.

If $g=h_6 xy^\ga$, then $xg=x(y^{\ep_k}v^{\del_{k-1}}y^{\ep_{k-1}} \dots v^{\del_1}y^{\ep_1})xy^\ga=v^{\ep_k}y^{\del_{k-1}}v^{\ep_{k-1}}\dots y^{\del_1}v^{\ep_1}y^\ga\in H$.\\
This shows that $x(\be)\in\al^H$. Similarly, if $x(\be)\not\in \al^H$, then $\be\in \al^H$.

Secondly, suppose that $\al^G$ contains a unique totally positive cycle $\stackrel{y}{\wh{\;\be\;}}$. By Lemma \ref{cycles} (iii), $\stackrel{v}{\,\wh{x(\be)\,}}$ is a totally negative cycle which is obviously contained in $\al^G$. Note also, by Lemma \ref{cycles} (iv), that it cannot be the case that both cycles are contained in $\al^H$. Following exactly the same argument of the zero-norm elements above, we conclude that $\al^G$ is either the disjoint union of $\al^H$ and $\be^H$ or is the disjoint union of $\al^H$ and $x(\be)^H$, and therefore $\al^G$ splits into two distinct $H$-orbits in this case as well.

We have, so far, proved that $|\mathcal{O}_{-n}^H|=2|\mathcal{O}_{-n}^G|$ for $n\geq 2$. The proof would by complete in this case by applying \cite[Theorem 2.1]{DA}, which precisely gives the value of $|\mathcal{O}_{-n}^G|$ as follows:
\begin{align*}
|\mathcal{O}^G_{-n}|
 &= \left \{ \begin{array} {c@{\,,\,}l} 2 & \mbox{if $n=2$}\\
 4 & \mbox{if $n=3$} \\
d(n)+ \frac{2}{3}\;\sum_{i=1}^{\lfloor (n-1)/2 \rfloor}[ d(i^2+n)-2d_{\leq i} (i^2+n)] & \mbox{otherwise.}
\end{array}\right.
\end{align*}
and, moreover, declares the congruence $|\mathcal{O}^G_{-n}| \equiv 0\;(\bmod \,4)$ for $n\geq 3$.

Finally, for the case $n=1$, the only elements of norm zero in $\Q^*(\sqrt{-1})$ are $\pm \sqrt{-1}$, and both are fixed by $x$ (see the proof of Lemma \ref{cycles} (v)). We also have by Lemma \ref{orbits lemma} (i) that $\al^G$ contains a unique element of norm zero; that is $\al^G$ must contain either $\sqrt{-1}$ or $- \sqrt{-1}$ but not both. Suppose that $\sqrt{-1} \in \al^G$. If $\sqrt{-1}\not\in \al^H$, then, by following the argument in the second paragraph of this proof, $x(\sqrt{-1})\in \al^H$. But $x(\sqrt{-1})=\sqrt{-1}$, a contradiction. Thus, $\sqrt{-1}\in \al^H$. Similarly, if $- \sqrt{-1}\in \al^G$, then $- \sqrt{-1}\in \al^H$. This shows that $\al^H=\al^G$ and hence $|\mathcal{O}^H_{-1}|=|\mathcal{O}^G_{-1}|$, which in turn equals 2 (by \cite[Theorem 2.1]{DA}).
\end{proof}
\vspace{.3cm}
\begin{remarks} We present below some examples showing that the formulas given by \cite[Theorem 9]{AM2} are incorrect in general (see Claim 3 in Section 2).

1. For $n=1$, the number of orbits $|\mathcal{O}_{-1}^H|$ given by \cite[Theorem 9]{AM2} is
$$2[d(1)+2d(2)-6]=2[1+4-6]=-2,$$
which is absurd. By Theorem \ref{orbits}, however, $|\mathcal{O}_{-1}^H|= 2$.

2. For $n=21$, the number of orbits $|\mathcal{O}_{-21}^H|$ given by \cite[Theorem 9]{AM2} is
$$2[d(21)+2d(22)-6]=2[4+8-6]=12.$$ This estimate immediately appears to be wrong when tested using the congruence relation in Theorem \ref{orbits}. Nonetheless, Theorem \ref{orbits} gives precisely the following estimate
\begin{align*}
|\mathcal{O}_{-21}^H|&= 2d(21)+\frac{4}{3}\,\sum_{i=1}^{10} \left[d(i^2+21)-2d_{\leq i}(i^2+21)\right]\\
                     &= 8+\frac{4}{3}\{d(22)-2d_{\leq 1}(22) + d(25)-2d_{\leq 2}(25) + d(30)-2d_{\leq 3}(30)+ d(37)-2d_{\leq 4}(37) \\
                     & \qquad\qquad + d(46)-2d_{\leq 5}(46) + d(57)-2d_{\leq 6}(57) + d(70)-2d_{\leq 7}(70) + d(85)-2d_{\leq 8}(85) \\
                      & \qquad\qquad + d(102)-2d_{\leq 9}(102) + d(121)-2d_{\leq 10}(121)\} \\
                     &= 8 + \frac{4}{3}\{4-2+3-2+8-6+2-2 + 4-4+4-4+8-8+4-4 +8-8+3-2\}\\
                     &=8+8\\ &=16.
\end{align*}

3. For $n=26$, the number of orbits $|\mathcal{O}_{-26}^H|$ given by \cite[Theorem 9]{AM2} is
$$2[d(26)+2d(27)-4]=2[4+8-4]=16.$$
Theorem \ref{orbits} gives precisely the following estimate
\begin{align*}
|\mathcal{O}_{-26}^H|&= 2d(26)+\frac{4}{3}\,\sum_{i=1}^{12} \left[d(i^2+26)-2d_{\leq i}(i^2+26)\right]\\
                     &= 8+\frac{4}{3}\{d(27)-2d_{\leq 1}(27) + d(30)-2d_{\leq 2}(30) + d(35)-2d_{\leq 3}(35) \\
                     & \qquad\qquad + d(42)-2d_{\leq 4}(42) + d(51)-2d_{\leq 5}(51) + d(62)-2d_{\leq 6}(62) \\
                      & \qquad\qquad + d(75)-2d_{\leq 7}(75) + d(90)-2d_{\leq 8}(90) + d(107)-2d_{\leq 9}(107) \\
                      &\qquad\qquad + d(126)-2d_{\leq 10}(126)+ d(147)-2d_{\leq 11}(147) + d(170)-2d_{\leq 12}(170)\} \\
                     &= 8 + \frac{4}{3}\{4-2+8-4+4-2+8-6 + 4-4+4-4+6-6+12-10 +2-2\\
                     &\qquad\qquad\; +12-12+6-6+8-8\}\\
                     &=8+16\\ &=24.
\end{align*}
\end{remarks}

\section{Conclusion}
Consider the subgroup $H=\langle y,\,v\,|\,y^3=v^3=1\rangle$ of $\mbox{PSL}(2,\mathbb{Z})=\langle x, \, y \,|\, x^2=y^3=1\rangle$, where $x: z\mapsto -1/z$, $y:z\mapsto (z-1)/z$, and $v=xyx: z \mapsto -1/(z+1)$. For a square-free $n\in \N$, it is shown that the number of orbits in the set
$$\Q^*(\sqrt{-n})=\{\frac{a+\sqrt{-n}}{c} \in \Q(\sqrt{-n}) \,|\, a,b=\frac{a^2+n}{c}, c \in \Z, c\neq 0\},$$
under the action of $H$, is given as follows:
$$\begin{cases}
2 & \mbox{, if $n=1$}\\
4 & \mbox{, if $n=2$}\\
8 & \mbox{, if $n=3$} \\
2d(n)+ \frac{4}{3}\;\sum_{i=1}^{\lfloor (n-1)/2 \rfloor}[ d(i^2+n)-2d_{\leq i} (i^2+n)] & \mbox{, if $n>3$}
\end{cases}$$
and is congruent to 0 modulo 8 for all $n\geq 3$, where $d(k)$ denotes the number of positive divisors of $k$, $d_{\leq i}(k)$ denotes the number of positive divisors of $k$ which do not exceed $i$ for $i,k\in \N$ with $i\leq k$, and $\lfloor \; . \; \rfloor$ is the floor function.


\end{document}